\documentclass[10pt]{amsart}
\usepackage{amsmath}
\usepackage{amsfonts}
\usepackage{amssymb}
\usepackage{amsthm}
\usepackage{url}
\usepackage{tikz-cd}
\usepackage{dsfont}
\usepackage{graphicx}
\usepackage{caption}
\usepackage{subcaption}
\usepackage{comment} 
\usepackage{hyperref}
\usepackage{todonotes}
\usepackage{color}
\usepackage{enumerate}

\numberwithin{equation}{section}

\newtheorem{proposition}{Proposition}[section]
\newtheorem{lemma}[proposition]{Lemma}
\newtheorem{theorem}[proposition]{Theorem}
\newtheorem{corollary}[proposition]{Corollary}

\theoremstyle{definition}
\newtheorem{remark}[proposition]{Remark}
\newtheorem{definition}[proposition]{Definition}

\DeclareMathOperator{\Aut}{Aut}

\DeclareMathOperator{\Dom}{Dom}
\DeclareMathOperator{\pr}{pr}
\DeclareMathOperator{\Vol}{Vol}

\newcommand{\ddb}{i\partial \bar\partial}

\def\R{\mathbb{R}}

\def\P{\mathbb{P}}

\def\C{\mathbb{C}}

\def\M{\mathbb{M}}

\def\cD{\mathcal{D}}

\def\cG{\mathcal{G}}
\def\cH{\mathcal{H}}

\def\cJ{\mathcal{J}}
\def\cX{\mathcal{X}}

\def\cW{\mathcal{W}}

\def\cZ{\mathcal{Z}}

\def\cC{\mathcal{C}}
\def\cA{\mathcal{A}}

\def\ocZ{\overline{\mathcal{Z}}}

\def\delb{\overline{\partial}}

\def\Isom{\mathrm{Isom}}
\def\Lie{\mathrm{Lie}}
\def\scal{\mathrm{scal}}
\def\dim{\mathrm{dim}}
\def\Int{\mathrm{Int}}

\def\k{\mathfrak{k}}

\def\t{\mathfrak{t}}

\def\Om{\Omega}
\def\om{\omega}
\def\ep{\varepsilon}

\def\>{\rangle}

\def\<{\langle}
\def\>{\rangle}

\def\uep{\underline{\ep}}

\title[Analytic $K$-semistability and local wall-crossing]{Analytic $K$-semistability and local wall-crossing}

\author[Lars Martin Sektnan and Carl Tipler]{Lars Martin Sektnan and Carl Tipler}

\address{Lars Martin Sektnan, Department of Mathematical Sciences, University of Gothenburg, 412 96 Gothenburg, Sweden}

\email{sektnan@chalmers.se}

\address{Carl Tipler, Univ Brest, UMR CNRS 6205, Laboratoire de Mathématiques de Bretagne
Atlantique, France}
\email{Carl.Tipler@univ-brest.fr}


\begin{document}

\begin{abstract} 
For a small polarised deformation of a constant scalar curvature K\"ahler manifold, under some cohomological vanishing conditions, we prove that $K$-polystability along nearby polarisations implies the existence of a constant scalar curvature K\"ahler metric.  In this setting, we reduce $K$-polystability to the computation of the classical Futaki invariant on the cscK degeneration. Our result holds on specific families and provides local wall-crossing phenomena for the moduli of cscK manifolds when the polarisation varies.
\end{abstract}

\maketitle

\section{Introduction}
One of the central themes of K\"ahler geometry is the search for canonical K\"ahler metrics. A natural notion of canonical metric is constant scalar curvature K\"ahler (cscK) metrics. Such metrics may exist in some, all or no K\"ahler classes and the goal is to determine which classes on a given manifold admit such metrics. The Yau--Tian--Donaldson conjecture (\cite{yau, tian, donaldson}) predicts that the existence of a cscK metric should be equivalent to an algebro-geometric notion of stability, called $K$-polystability. The conjecture has had great progress, but still remains open.

Classical results of LeBrun--Simanca (\cite{lebrunsimancaGAFA}) say that a cscK manifold will admit a cscK metric in all nearby K\"ahler classes, and on all nearby complex deformations, provided the reduced automorphism group of $X$ is trivial. In general, while (\cite{lebrunsimanca93}) says that a cscK manifold will still admit a more general type of canonical metric, called extremal K\"ahler metrics, in all nearby K\"ahler classes, the presence of automorphisms may prevent such openness when deforming the complex structure (\cite{donaldson-largesym,gabor-deformations,rt}). 

Our work focuses on the YTD conjecture when we 
vary both the K\"ahler class and the complex structure simultaneously, in the presence of symmetries. Our setting is working with polarised K\"ahler manifolds $(X, [\omega])$ that are small deformations of a cscK manifold $(X_0,\omega)$, where by a polarised K\"ahler manifold we mean the pair of a K\"ahler manifold with a fixed K\"ahler class. This in particular implies that $(X,[\omega])$ is $K$-semistable (\cite{donaldson-lower}), and covers as a special case of interest \emph{analytically $K$-semistable} manifolds, when $(X_0, \om)$ is the central fiber of a regular {\it test configuration} for $(X, [\om])$ (see Section \ref{sec:backgroundKstability} for the definitions). In the latter case, unless it is isomorphic to $X_0$, $X$ will then not admit a cscK metric in $[\omega]$ (\cite{tian,stoppa,dervan-relative}). 
We then investigate what happens as one varies this K\"ahler class, where $(X,[\omega])$ may or may not admit a cscK metric.
\begin{theorem}
\label{thm:main}
Suppose $(X, [\omega])$ is a small polarised complex deformation of a cscK manifold $(X_0, [\om])$.  Assume that $h^{0,1}(X_0)=h^{0,2}(X_0)=0$ and $H^2(X_0,TX_0)=0$. Then for any class $[\alpha] \in H^2(X, \R) \cap H^{1,1}(X)$, there exists an $\varepsilon_0 > 0$ such that for all $\varepsilon \in (-\varepsilon_0, \varepsilon_0)$, $(X,[\omega + \varepsilon \alpha])$ is $K$-polystable if and only if $X$ admits a cscK metric in $[\omega+\varepsilon \alpha]$.  Moreover, to test for $K$-polystability in these classes, it suffices to check positivity of the Futaki invariant of certain holomorphic vector fields on $X_0$.
\end{theorem}
We refer to Section \ref{sec:invariantsandwallcrossing} for the precise description of the positivity conditions required on the Futaki invariant of $X_0$ for Theorem \ref{thm:main} to apply.

We should emphasize that while deformations of cscK manifolds through perturbations of complex structures or K\"ahler classes have been already studied (\cite{lebrunsimanca93, gabor-deformations}), the results already at hand cannot be used to obtain Theorem \ref{thm:main}. Indeed, in our setting (assuming $\ep \neq 0$ and that $X$ is not isomorphic to $X_0$), one starts with a non-cscK manifold $(X,[\om])$, and tries to perturb it to a cscK one by moving the K\"ahler class to $[\om+\epsilon \alpha]$. As $(X,[\om])$ is not cscK, results from \cite{lebrunsimanca93} do not apply, and while $(X_0, [\om])$ is cscK, $(X_0, [\om+\epsilon \alpha])$ is not, and then \cite{gabor-deformations} doesn't apply either to conclude the existence of a cscK metric on $(X,[\om+\epsilon\alpha])$.

The main theorem is a local verification of the Yau--Tian--Donaldson conjecture, near a cscK manifold. The key technical difficulty to overcome is that we are in general perturbing from a strictly semistable manifold instead of one that is  polystable. The present work provides a technique that seems quite general in order to address such problems on the analytic side. Moreover, the final condition really says that the verification of $K$-polystability is a finite condition in these classes, which we now explain.

In general, $K$-polystability asks one to verify that a certain numerical invariant, the {\it Donaldson--Futaki invariant} (see Section \ref{sec:backgroundKstability}), is non-negative for all test configurations of $X$, with equality if and only if the test configuration is a product test configuration, constructed from a $\C^*$-action on $X$. Since $X_0$ admits a cscK metric, it is $K$-polystable (\cite{stoppa,dervan-relative}), and appears as a most destabilizing object for $X$. It is then natural to expect that the classical Futaki invariant of $X_0$ will encode $K$-polystability for $X$ in nearby classes.

The precise test configurations needed to test $K$-polystability of $X$ are obtained from the action of the reduced automorphism group of $X_0$ on its Kuranishi slice, and the associated invariants to be computed are explicitly described in Section \ref{sec:invariantsandwallcrossing}. We should highlight that it is a very difficult problem to test $K$-polystability in general, and our work provides a new instance where this can be achieved with a finite number of explicit test configurations. As an illustration, in Section \ref{sec:examples}, we apply Theorem \ref{thm:main} to small deformations of a $4$-point blow-up of $\P^1\times \P^1$, completing the picture developed in \cite{ClTip,rt}.

Our result exhibits some local wall-crossing phenomena for $K$-stability, as one varies the polarisation. Such phenomena arise when a moduli space parametrising families of stable objects undergo transformations when the stability condition varies. As a leading example, in \cite{DolHu}, the cone of linearisations for a given reductive group action is divided into chambers, where crossing a wall results in a birational transformation of the associated GIT quotient. While there are deep results on wall-crossings for moduli of stable bundles on surfaces (see \cite[Chapter 4, Section C]{HuLe} and reference therein), moduli of curves (e.g. \cite{HaHy}), or moduli of log Fano pairs (see e.g. \cite{AsDeLiu}), far less seems to be known about variations of cscK moduli, as built in \cite{fujikischumacher,dervan-moduli}. While a broader approach for wall-crossings of moduli of varieties is developed in \cite{dervanstabilityconditions}, in this work we focus on local variations of the cscK moduli under perturbations of the polarisation. As we explain in Section \ref{sec:invariantsandwallcrossing}, our result holds uniformly in specific families of polarised manifolds, and shows that the sign of a finite number of Futaki invariants on $X_0$ governs whether or not those families will contribute to the corresponding cscK moduli. 
\begin{remark}
It would be very interesting to address the issues and results of the present work purely algebro-geometrically, for example using the newly developed non-Archimedean framework for $K$-stability of Boucksom--Jonsson (\cite{boucksom-jonsson}). In particular, they prove the (uniform) $K$-stability analogue of the LeBrun--Simanca openness theorem (\cite{lebrunsimanca93}). Note, however, that it is crucial to allow for automorphisms in the current situation. Indeed, while the strictly $K$-semistable manifold might have discrete automorphism group, the $K$-polystable central fibre it is a small deformation of will \emph{always} have non-trivial reduced automorphism group. The work of Boucksom--Jonsson is under the assumption that the reduced automorphism group is trivial and thus their techniques cannot immediately be applied to give purely algebro-geometric counterparts to our results. On the other hand, with an extension of their work to allow for automorphisms, it is more likely their techniques also work in situations where the $K$-polystable central fibre is singular. 
The above is thus one potential application of extending their techniques to allow for non-trivial reduced automorphism groups. 
Note however that due to the lack of the resolution of the Yau--Tian--Donaldson conjecture in general, this would not give the same implications about existence of cscK metrics.
\end{remark}

Finally, let us describe the main steps in the proof of Theorem \ref{thm:main}. 
The first one is 
the construction of a family of relative K\"ahler forms parametrised by $\ep$ on the total space $\cX$ of Kuranishi's semi-universal family of polarised deformations of $(X_0, [\om])$. This requires a cohomological vanishing result, that we achieve in Section \ref{sec:cohomologyvanishing}, using Hodge theory for Dolbeault cohomology on manifolds with boundary, and this is where the hypotheses $h^{0,1}=h^{0,2}=0$ are used.  With this at hand, we reduce the problem to finding zeros of a finite dimensional moment map as in Ortu's work \cite{ortu24}, using Dervan and Hallam's framework \cite{dervanhallam23}, with the additional difficulties coming from the perturbation of the K\"ahler class. We should stress that we borrowed the idea of perturbing finite dimensional moment maps to solve a perturbative problem in K\"ahler geometry from \cite{dervanstabilityconditions,dervan-sektnan-blowups}. Then, we present two proofs to solve the resulting finite dimensional problem. The first one requires the automorphism group of $X_0$ to be abelian (see Section \ref{sec:finitedimproblem}) and is fairly different from the approaches in \cite{dervanstabilityconditions,dervan-sektnan-blowups} as it does not use the moment map flow.  Instead, we use a more geometric approach in Section \ref{sec:localrigidity} where we develop a local rigidity result for families of moment maps. This clearly and explicitly produces the numerical invariants needed to test stability. This new method seems to be quite robust, and we give a general and abstract presentation of this section, that could be of independent interest for other perturbative problems related to infinite dimensional moment maps. Then, in Section \ref{sec:solve general}, we prove Theorem \ref{thm:main} without any hypothesis on the automorphism groups. There, we follow closely the methods used in \cite{ortu24,ortusektnan24}.

\begin{remark}
 \label{rem:previous approach}
 The current version of the article has a completely different approach than our initial version. In the initial version, we perturbed the Kuranishi slice by using Moser's trick. The scalar curvature under those perturbations is only a pseudo-differential operator, leading to technical difficulties that were pointed out by the referee.
\end{remark}

\subsection*{Acknowledgments}  
The authors would like to thank Ruadha\'i Dervan for pointing several relevant references out to us, and for interesting discussions on the topic. We also thank Jacopo Stoppa and him for pointing out to us that the stability criterion we obtain can always be computed on the central fibre. Finally, we would like to thank Bo Berndtsson who suggested the book of Folland--Kohn to tackle Section \ref{sec:cohomologyvanishing}. We also thank the anonymous referee for their many useful comments and who pointed out a gap in the previous version, cf remark \ref{rem:previous approach}. CT is partially supported by the grants MARGE ANR-21-CE40-0011 and BRIDGES ANR--FAPESP ANR-21-CE40-0017. LMS is funded by a Marie Sk\l{}odowska-Curie Individual Fellowship, funded from the European Union's Horizon 2020 research and innovation programme under grant agreement No 101028041, and a Starting Grant from the Swedish Research Council (grant 2022-04574).

\section{Extremal metrics, stability and deformations}
\label{sec:background}
In this section, we will introduce notations and recall classical results about extremal K\"ahler metrics and $K$-stability. We refer the reader to \cite{gauduchon} and \cite{gabor-book} for a detailed treatment.

\subsection{Extremal metrics and the moment map picture}
\label{sec:backgroundextremal}

Let $(X_0, [\om])$ be a polarised K\"ahler manifold. 
We denote $M$ the underlying smooth manifold and $J_0$ the almost complex structure, so that $X_0=(M, J_0)$. Recall that $\om$ is said to be \emph{extremal} if its scalar curvature function is a hamiltonian (with respect to $\om$) of a holomorphic vector field. Equivalently, 
$\mathrm{scal(J_0,\om)}$ is in the kernel of the operator
\begin{equation}
 \label{eq:Doperator}
\begin{array}{cccc}
 \cD:  & \cC^\infty(M,\R) & \to & \Om^{0,1}(T^{1,0}) \\
           &   f &\mapsto & \delb (\nabla_{\om}^{1,0} f),
\end{array}
\end{equation}
where $\nabla_{\om}^{1,0} f$ is defined by $\delb f = \om( \nabla_{\om}^{1,0}f , \cdot )$. The kernel of $\cD$ is the Lie algebra of $\Isom_0(M,J_0,\om)$, the connected component of the group of hamiltonian biholomorphism of $(M,J_0,\om)$. A special case of extremal metrics are given by constant scalar curvature metrics (cscK metrics for short), for which, by Matsushima and Lichnerowicz, the {\it reduced automorphism group} is the complexification of the compact Lie group $\Isom_0(M,J_0,\om)$. 

Independently, Fujiki (\cite{fujiki}) and Donaldson (\cite{donaldson-moment}) provided an infinite dimensional moment map interpretation of cscK metrics (see \cite[Chapter 9]{gauduchon} for a detailed treatment including the case of extremal metrics). Let $\cJ$ be the space of almost-complex structures on $M$ and denote $\cJ_\om$ the subspace of $\om$-compatible elements. The group $\cG_\om$ of hamiltonian symplectomorphisms acts on $\cJ_\om$ via the standard action of diffeomorphisms on $\cJ$, and $\cJ_\om$ carries a formal symplectic structure $\Om_{\cJ_{\om}}$ such that the $\cG_\om$-action is hamiltonian, with moment map 
$$
\begin{array}{cccc}
 \mu_{\cJ_{\om}} : & \cJ_{\om} & \to & \cC^\infty_0(M,\R)\\
                   &   J     & \mapsto & \mathrm{scal}(J,\om)-\hat s.
\end{array}
$$
Here $\hat s$ is the average of the hermitian scalar curvature $\mathrm{scal}(J,\om)$ of $(J,\om)$ and $\cC^\infty_0(M,\R)$ is the space of mean value zero function (with respect to the volume form $\om^n$), identified with the dual of the Lie algebra of $\cG_\om$ via the pairing 
\begin{equation}
 \label{eq:FutMabpairing}
\langle f , h \rangle_\om = \int_M fg\; \frac{\om^n}{n!}.
\end{equation}
While the complexification of $\cG_\om$ does not exist, a key fact is that one can make sense of the leaf of a ``$\cG_\om^\C$-foliation'' in the following manner. The operator $\cD$ gives the infinitesimal action of $\cG_\om$ on $T_{J_0}\cJ_\om$. We can extend  this operator $\C$-linearly, and we will say that $J_1$ is in the $\cG_\om^\C$-orbit of $J_0$ if there are paths $(J_t)_{t\in[0,1]}$ in $\cJ_\om$ and $(\phi_t)_{t\in[0,1]}$ in $\cC_0^\infty(M,\C)$ such that
$$
\frac{d}{dt} J_t = \cD_t \phi_t,
$$
where $\cD_t$ is the operator \eqref{eq:Doperator} at $J_t$. Locally, we can consider small variations of almost complex structures along ``complexified orbits'' as follows.

For any $\phi\in \cC_0^\infty(M,\R)$ sufficiently small in some appropriate $L^{2,k}$-norm, we can, using Moser's trick, find a diffeomorphism $F_{\om,\phi}$ such that $F_{\om,\phi}^*(\om-dJ_0d\phi)=\om$. We then set 
\begin{equation}
 \label{eq:complexifiedorbit}
J_{\phi} = F_{\om,\phi}^*J_0.
\end{equation}
We then see that varying in a complexified orbit in $\cJ_\om$ amounts to varying the K\"ahler form in a K\"ahler class for a fixed complex structure, and we will take the former point of view to look for cscK (or extremal) metrics within a fixed K\"ahler class.

\subsection{$K$-stability}
\label{sec:backgroundKstability}

The Yau--Tian--Donaldson conjecture predicts that an extremal metric should exist on $(X_0,[\om])$ if and only if it satisfies a relative $K$-stability condition (\cite{yau, tian, donaldson, gabor-thesis}). While recent progress shows that the right notion to consider should be a uniform version of $K$-stability, we will only need a weaker notion in the present work.
\begin{definition}
A \emph{regular test configuration} for a polarised K\"ahler manifold $(X,[\om])$ is given by 
a relatively polarised $\C^*$-equivariant family of K\"ahler manifolds $$\pi: (\cX,\cA)\to \C$$ for the natural $\C^*$-action on $\C$ with $(\cX_1,\cA_1)\simeq (X,[\om])$. Such a test configuration is called a \emph{product test configuration} if $(\cX,\cA)$ is $\C^*$-equivariantly isomorphic to $(X\times \C, [\om_0])$, with $\C^*$-action induced by a one-parameter subgroup of the reduced automorphism group of $(X,\om)$ and the standard action on the $\C$ factor.
\end{definition}
In this definition, by {\it family} we mean that the map $\pi$ must be surjective and a proper submersion. For any such test configuration $(\cX, \cA)$, the $\C^*$-action on the central fiber $(\cX_0,\cA_0)$ provides the Donaldson--Futaki invariant :
$$
\mathrm{DF}(\cX,\cA):=-\int_{\cX_0} H_u\,(\mathrm{scal}(J_0,\om_0)-\hat s)\; \frac{\om_0^n}{n!},
$$
where $(M, J_0, [\om_0])\simeq (\cX_0,\cA_0)$, and where $H_u$ is the hamiltonian of the generator $u\in\Gamma(T^{1,0}\cX_0)$ of the $S^1$-action on the central fiber with respect to $\om_0$.

\begin{definition}
 We will say that $(X, [\om])$ is \emph{$K$-polystable} if for all regular test configurations $(\cX,\cA)$ for $(X,[\om])$, the associated Donaldson--Futaki invariant satisfies
 $$
 \mathrm{DF}(\cX,\cA)\geq 0,
 $$
 with equality only for product test configurations.
\end{definition}
\begin{remark}
$K$-polystability usually requires one to also consider test configurations that are not regular. In general, this is needed, but we will only require this weaker notion of $K$-polystability in the present work.
\end{remark}

\subsection{Deformation theory}
\label{sec:deformationsandslice}
In \cite[Section 2]{fujikischumacher}, a deformation complex for the pair $(X_0, [\om])$ is introduced. First, the usual Kuranishi complex for deformation of complex structures on $X_0=(M, J_0)$ is (\cite{kuranishi})
$$
\ldots \to \Om^{0,k}(T^{1,0}) \stackrel{\delb}{\to} \Om^{0,k+1}(T^{1,0}) \to \ldots ,
$$
where the extension of $\delb$ to $(0,p)$-forms is given in local coordinates, for $\beta = \sum_j \beta_j \otimes \frac{\partial}{\partial z^j} $ , by 
$$
\delb \beta = \sum_j \delb \beta_j \otimes \frac{\partial}{\partial z^j}.
$$
We wish to consider deformations that are compatible with $\om$. Following \cite{fujikischumacher}, we define maps
$$
\begin{array}{cccc}
\iota^k_\bullet  : & \Om^{0,k}(T^{1,0}) & \to &\Om^{0,k+1}\\
                &  \beta & \mapsto & \iota_\beta\om
\end{array}
$$
where $\iota_\beta \om$ is obtained by the composition of first contraction and then alternation operators. For $k\geq 1$, we then set :
$$
\Om_{\om}^{0,k}(T^{1,0}) := \ker \iota^k_\bullet,
$$
while we define
$$
\Om_{\om}^{0,0}(T^{1,0}):= \cC^\infty(M,\C).
$$
Together with the restriction of $\delb$, and the map $\delb^0:=\cD$ defined by (\ref{eq:Doperator}),  we obtain an elliptic complex $(\Om_{\om}^{0,\bullet}(T^{1,0}), \delb^\bullet )$.
We denote by $H^{0,k}_{\om}(T^{1,0})$ the associated cohomology groups. Then, from this complex, one can build a slice for the $\cG^\C$-action on $\cJ_\om$ around $J_0$. As in \cite{doan}, setting $K:=\Isom_0(M, J_0,\om)$, this construction can actually be made $K$-equivariantly, and even locally $K^\C$-equivariantly, for the natural $K^\C$-actions on $H^{0,1}_{\om}(T^{1,0})$ and $\cJ_\om$ induced by $K\subset \cG_\om\subset \mathrm{Diff}(M)$. To sum up: 
\begin{proposition}
 \label{prop:firstslice}
 There exists a holomorphic $K$-equivariant map 
 $$\Phi : B \to \cJ_\om$$
 from a ball around the origin $B\subset H^{0,1}_{\om}(T^{1,0})$ such that 
 \begin{enumerate}
  \item $\Phi(0)=J_0$;
  \item $Z:=\Phi^{-1}(\cJ_\om^{\mathrm{int}})$ is a complex subspace of $B$;
  \item if $b,b'\in Z$ lie in the same $K^\C$-orbit, then $(M,\Phi(b))$ and $(M, \Phi(b'))$ are isomorphic complex manifolds;
  \item The $\cG^\C$-orbit of any small complex deformation of $X_0$ polarised by $[\om]$ intersects $\Phi(Z)$.
 \end{enumerate}
\end{proposition}
We refer to \cite[Lemma 6.1]{ChenSun} for a detailed proof of this proposition.  We will assume from now on that $H^2(X_0,TX_0)=0$ so that  $Z=B$ in Proposition \ref{prop:firstslice}. From deformation theory (\cite[Chapter 8]{kuranishibook}), this map induces a $K$-equivariant semi-universal family of polarised deformations of $(X_ 0, [\om])$ 
$$
\pi : \cX\to B
$$
such that any two points $(b, b')$ in the same $K^\C$-orbit on $B$ correspond to isomorphic polarised manifolds $(\cX_b,[\om])\simeq (\cX_{b'}, [\om])$. We will also assume that  $h^{0,2}(X_0)=0$, so that for any small complex deformation $X$ of $X_0$, one has $h^{0,2}(X)=0$ (\cite[Chapter 9]{voisin1}) and 
$$
H^{1,1}(X,\R) = H^2(M, \R) = H^{1,1}(X_0, \R),
$$
where we denote $H^{1,1}(X,\R):=H^{1,1}(X,\C)\cap H^2(X,\R)$.  Let $\cH^{1,1}(X_0,\R)$ be the space of $(\om,J_0)$-harmonic real $(1,1)$-forms on $(X_0, \om)$ and let $( \alpha_1, \ldots, \alpha_r)$ be a basis of this vector space. Note that by uniqueness of harmonic representatives of Dolbeault cohomology classes, the $\alpha_i$'s are $K$-invariant. For any $\uep:=(\ep_1,\ldots, \ep_r)$, we set 
$$
\om_{\uep}:=\om+\sum_{i=1}^r \ep_i \alpha_i.
$$
This defines a K\"ahler form on $X_0$ for all $\uep$ in a sufficiently small neighbourhood $U$ of zero in $\R^r$. For any $b\in B$ and $\uep\in U$, we have $[\om_{\uep}]\in H^{1,1}(\cX_b,\R)$. However, $\om_{\uep}$ itself might not be of type $(1,1)$ with respect to $\Phi(b)$. We wish to correct this, and to produce a family of relative K\"ahler forms $\om'_{\uep}$ on the total space $\cX$ of the family whose cohomology class restricts to $[\om_{\uep}]$ on each fiber of $\cX\to B$. To perform this, 
we first need to show a cohomology vanishing result for $\cX$.

\section{Cohomology vanishing}
\label{sec:cohomologyvanishing}

To apply our strategy we will need to use a good globally $(1,1)$ representative of the relevant cohomology classes on each fibre of the universal family $\cX \to B$. In order to achieve this, we will apply a Hodge theory in Dolbeault cohomology for complex manifolds with boundary. The main goal of this section is to obtain a cohomology vanishing result on $\cX$ for this cohomology, see Lemma \ref{lem:vanishing02}.

\subsection{Hodge theory for Dolbeault cohomology on manifolds with boundary}
We now explain the relevant part of this theory that we will need, see \cite{follandkohn} for this and much deeper results on this topic. Let $Y$ be a complex manifold with boundary $\partial Y$ sitting inside some larger ambient complex manifold $Y'$. More precisely, assume $\partial Y$ is given by the equation $R=0$, where $R$ is a smooth function such that $R<0$ on $Y$, $R>0$ on $Y' \setminus \bar Y$ and $dR \neq 0$ at any point in $\partial Y$. Assume that $Y'$ is K\"ahler and fix a K\"ahler metric $\omega_Y$ on $Y$ (which we assume is the restriction of a K\"ahler metric on $Y'$). In our application, we will take $Y$ to be the preimage $\pi^{-1}(B_1)$ of an open ball in $B$, which we without loss of generality can assume has radius $1$ with respect to the Euclidean metric, and $Y'$ to be the preimage of a slighly larger ball. $R$ is then given by $\pi^*(\varrho-1)$, where $\varrho$ is the radius function on $\C^n$. We will also take the K\"ahler metric to be the metric induced by $\omega + \pi^*(\omega_{Euc})$, which we can assume is K\"ahler on all of $\cX$ after shrinking $B$. Note that while the form $\omega + \pi^*(\omega_{Euc})$ is a product with respect to the smooth decomposition $\cX \cong B \times M$, the corresponding metric is not, as the holomorphic structure on $\cX$ is not a product. 

Going back to the general setup,  let $\Omega^{p,q}(\bar Y)$ denote the subspace of $\Omega^{p,q} (Y)$ consisting of elements that can be extended smoothly to $Y'$ and let $\delta : \Omega^{p,q}(Y) \to \Omega^{p,q-1}(Y)$ denote the formal adjoint of $\bar \partial$, defined by requiring 
$$
\int_Y (\delta (\alpha), \beta)_{\omega_Y} = \int_Y (\alpha, \bar \partial \beta)_{\omega_Y}
$$
for all $\beta$ with compact support in $Y$. Let $L^2(\Omega^{p,q}(\bar Y))$ denote the $L^2$-completion of $\Omega^{p,q}(\bar Y)$ with respect to the pairing
$$
\langle \cdot,\cdot \rangle = \int_Y (\cdot,\cdot)_{\om_Y}
$$
induced by $\omega_Y$. The operator $\bar \partial$ extends to an operator $L^2(\Omega^{p,q}(\bar Y)) \to L^2(\Omega^{p,q+1}(\bar Y))$ by taking the closure of the graph of $\bar \partial$ on $\Omega^{p,q}(\bar Y)$. 

There is then also an adjoint operator $\bar \partial^*$ to $\bar \partial$ on $L^2(\Omega^{p,q}(\bar Y))$. Its domain consists of the $\alpha \in L^2(\Omega^{p,q}(\bar Y))$ such that there exists a $c>0$ such that $|\langle \alpha, \bar \partial \tau \rangle| \leq c\|\tau\|$ for all $\tau \in \Omega^{p,q-1}(\bar Y)$. For such $\alpha$, the map $\tau \mapsto \langle \alpha, \bar \partial \tau \rangle$ is a bounded linear operator, and so is represented by some $\bar \partial^* \alpha \in L^2(\Omega^{p,q-1}(\bar Y))$, by the Riesz representation theorem. Note that as we are on a manifold with boundary, $\delta$ differs from the Hilbert space adjoint $\bar \partial^*$ on $\Omega^{p,q}(\bar Y)$ and Sobolev space completions, due to the presence of boundary terms when integrating by parts.

For $\alpha \in \Omega^{p,q}(\bar Y)$, let $\sigma_{\alpha} \in \Omega^{p,q-1}(\bar Y)$ be such that for each $y \in \bar Y$, 
$$(\beta_y, \sigma_{\alpha})_{\om_Y} = ((dR)^{0,1} \wedge \beta_y, \alpha)_{\om_Y}$$ for all $\beta_y \in \Lambda^{p,q-1}(T_y\bar Y)$, where we are using the pointwise inner product.
\begin{definition}
With $\bar Y \subset Y'$ and the defining function $R$ of $\partial Y$ as above, we say that a $(p,q)$-form $\alpha$ satisfies the \emph{$\bar \partial$-Neumann conditions} if on the boundary of $\bar Y$ we have 
\begin{itemize}
\item $\sigma_{\alpha} =0$;
\item $\sigma_{\bar \partial \alpha} = 0$.
\end{itemize}
\end{definition}

The first part of the $\bar \partial$-Neumann conditions has the following interpretation.
\begin{proposition}[{\cite[Prop. 1.3.2]{follandkohn}}]
We have that 
$$
\textnormal{Dom}(\bar \partial^*) \cap \Omega^{p,q}(\bar Y) = \{ \alpha \in \Omega^{p,q}(\bar Y) : \sigma_{\alpha}\big|_{\partial Y} = 0 \}
$$
and on this subspace $\bar \partial^* = \delta.$
\end{proposition}

Forms satisfying the $\bar \partial$-Neumann conditions in a weak sense form the domain of the so-called \emph{Friedrichs operator} $F$. This is the Hilbert space extension of the operator $I + \Delta$ where $\Delta$ is the $\bar \partial$-Laplacian on $(p,q)$-forms with compact support in $Y$. We will let $\mathcal{H}^{p,q}(\bar Y)$ denote the kernel of $F-I$ with this domain. Since we can apply the adjoint in the usual manner for this operator, we have that $\mathcal{H}^{p,q}(\bar Y) = \ker(\bar \partial) \cap \ker(\bar \partial^*)$ as usual (\cite[p. 18]{follandkohn}).

The following is the outcome of \cite[Chapter I]{follandkohn}.
\begin{proposition}
\label{prop:weakHodge}
There is a weak Hodge decomposition
$$
L^2(\Omega^{p,q}(\bar Y)) = \overline{\left(\bar \partial \delta (\Dom (F))\right)} \oplus \overline{\left(\delta \bar \partial (\Dom (F))\right)}  \oplus \mathcal{H}^{p,q}(\bar Y),
$$
where in the first two components we are taking the closure.
\end{proposition}
\begin{remark}
\label{rem:basicestimate}
In \cite[Prop. 3.1.12]{follandkohn}, Folland and Kohn assume the ``basic estimate" holds. Then there is a genuine decomposition 
$$
L^2(\Omega^{p,q}(\bar Y)) = \bar \partial \delta (\Dom (F)) \oplus \delta \bar \partial (\Dom (F))  \oplus \mathcal{H}^{p,q}(\bar Y)
$$
where no closure need to be taken in the first two components. The basic estimate holds for pseudoconvex domains in $\C^n$, and so therefore on $B$, but it is not clear to the authors that the basic estimate holds on the universal family $\cX$. However, the weak Hodge decomposition will be sufficient for our purposes.
\end{remark}

A contrasting feature for manifolds with boundary compared to the case of complex manifolds without boundary is that usual Schauder estimate fails near the boundary. However, on the interior, this estimate works as usual (\cite[Theorem 2.2.5]{follandkohn}). In particular, this implies  that if $\alpha \in \overline{\left(\bar \partial \delta (\Dom (F))\right)}$ and is smooth, then on a $Y'' \subset Y$ with compact closure we have that $\alpha = \bar \partial \beta$ for some smooth $(p,q-1)$-form $\beta$, see \cite[Theorem 2.2.9]{follandkohn}. In particular this holds for all $\bar \partial$-closed $\alpha$ satisfying the boundary condition that are orthogonal to $\mathcal{H}^{p,q}(\bar Y)$.

\subsection{Vanishing of the $\bar \partial$-cohomology}
We now apply the above theory to our special case of interest. Our goal is to show that for our choice of $Y=\cX$, then under our hypotheses we have the vanishing of the relevant cohomology group. We now let $\cX', B'$ be our previous $\cX,B$ and shrink $\cX,B$ slightly so that $B$ is given by a ball.
\begin{lemma}
\label{lem:vanishing02}
Suppose $H^{0,1}(\cX_0)=0$ and $H^{0,2}(\cX_0)=0$. Then $\mathcal{H}^{0,2}(\bar \cX) =0$.
\end{lemma}

Given this, any $\bar \partial$-closed $(0,2)$-form on $\cX$ will be $\bar \partial$-exact (at least on a smaller $\cX$).  To prove Lemma \ref{lem:vanishing02}, we will use the vanishing of all the higher cohomology on $\bar B$:
\begin{lemma}[\cite{folland72}]
\label{lem:PLball}
Let $B_1$ be the unit ball in $\C^m$. Then $\mathcal{H}^{p,q} (\bar B_1)=0$ for all $p+q>0$.
\end{lemma}

We now apply this to show that $\mathcal{H}^{0,2}(\bar \cX) = 0$. Note that as $\cX$ is $B \times M$ as a smooth manifold, we have a splitting $T\cX = TB \oplus TM$ and for all tensor bundles. In fact, this agrees with the symplectic splitting induced by the fibration structure $\cX \to B$ with relative symplectic from $\omega$, where the splitting is the $\omega$-orthogonal complement to the kernel $TM$ of derivative of the projection map, since $\omega$ is pulled back from the projection to the $M$-factor. However, note that as the splitting is not preserved by the almost-complex structure, this is not an orthogonal splitting with respect to the induced metric on $T\cX$. For every $b \in B$, we have an almost-complex structure $J_b$ on $M$ and $J_b(TM) = TM$ for all $b \in B$. However, $J_b$ does not preserve $TB$ in the above splitting. As the metric on $\bar \cX$, we will take $\omega + \omega_{Euc}$. With respect to the splitting of $T\cX$, $\omega$ is purely vertical and $\omega_{Euc}$ is purely horizontal. Thus the volume form is 
$$
\mathrm{dvol} = (\omega+\omega_{Euc})^{m+n} = \omega^n \wedge \omega_{Euc}^m.
$$
We will denote by $(z_i)_{1\leq i\leq m}$ holomorphic coordinates on $B$, so that the $d\bar z_i$'s stand for horizontal $(0,1)$-forms on $\cX$ (we omit the pullback in the notation).

We first show vanishing of the $(0,1)$-cohomology.
\begin{lemma}
\label{lem:vanishing01}
Suppose $H^{0,1}(\cX_0) = 0$. Then $\mathcal{H}^{0,1}(\bar \cX) = 0$.
\end{lemma}
\begin{proof}
Note that the hypothesis $H^{0,1}(\cX_0)=0$ implies $H^{0,1}(\cX_b)=0$ for each $b\in B$ (up to shrinking $B$). Suppose $\tau \in \mathcal{H}^{0,1}(\bar \cX).$ Then $\tau$ is $\bar \partial$-closed on each fibre $\cX_b$ for $b \in B$ and so by the vanishing of the $H^{0,1}(\cX_b)$, we can for each $b$ find an $f_b$ such that $\bar \partial_b f_b = \tau\big|_{\cX_b}$. Since the complex structure varies smoothly with $b$ we can ensure that the $f_b$'s glue to a smooth function $f$ on $\bar \cX$. This function then satisfies $\tau\big|_{\cX_b} = (\bar \partial f)\big|_{\cX_b}$ and so $\delb f$ and $\tau$ differ by a horizontal term. Since the ``basic estimate" discussed in Remark \ref{rem:basicestimate} may not hold on $\Omega^{0,1}(\bar \cX)$, we cannot conclude that $\tau$ is smooth up to the boundary. However, via cut-off functions, we can then create a sequence of functions $f_i$ such that $\bar \partial f_i$ converges to some $\zeta \in \Omega^{0,1}(\bar \cX)$ and such that $\tau - \zeta$ is purely horizontal, i.e.
$$
\tau = \zeta + \gamma
$$
where $\gamma = \sum_i h_i d\overline{z}_i$ for some functions $h_i$. Since $\tau$ and $\zeta$ are $\bar \partial$-closed, so is $\gamma$. In particular, since the only mixed component of $\bar \partial \tau$ is $\sum_i \bar \partial_b(h_i) \wedge d\overline{z}_i$, we have $\bar \partial_b(h_i) =0$ for all $i$ and $b$, and so the $h_i$, and hence $\gamma$, are pulled back from the base $B$.

Next, we make a suitable change in $\zeta$ to use that $\tau$ is co-closed. Consider the functional $\Phi$ on $\Omega^{0,0}(\bar B)$ given by
$$
h \mapsto \langle \delb^*\zeta, \pi^*h \rangle_{\bar \cX}.
$$
This is the represented as 
$$
\Phi(h) = \langle \tilde{g},h \rangle_{\bar B}
$$
for some $\tilde{g} \in L^2(\Omega^{0,0}(\bar B))$. Moreover, since $\Phi$ vanishes on constant functions, $\tilde{g}$ has average $0$, and so is in the image of the Laplacian on the base. Thus
$$
\Phi(h) = \langle \bar \partial g', \bar \partial h \rangle_{\bar B}
$$
for some function $g'$ on $\bar B$. 

Now, note that 
$$
\langle \bar \partial g', \bar \partial h \rangle_{\bar B} = \frac{1}{\Vol(\cX_0)} \langle \bar \partial \pi^*g', \bar \partial \pi^*h \rangle_{\bar \cX}
$$
since the induced product on forms on $\bar \cX$ restricts to the Euclidean inner product on forms pulled back from $\bar B$, and the fibre volume is independent of the fibre (here we use the unique continuous extension of $\pi^*$ to the Sobolev $L^2$ spaces under consideration). If we then let $g = \frac{\pi^*g'}{\Vol(\cX_0)}$, we have 
$$
\langle \zeta - \bar \partial (g) , \bar \partial \pi^*h \rangle_{\cX} =0 
$$
for all $h$ on $\bar B$. But this is equivalent to
$$
\langle \bar \partial^* (\zeta - \bar \partial (g)), \pi^*h \rangle_{\cX} =0
$$
for all $h$. If we replace $\gamma$ by $\gamma' = \gamma+ \bar \partial g$, then $\gamma'$ is still pulled back from a $(0,1)$-form $\gamma''$ on the base. Since $\bar \partial^* \tau = 0$ we then see that for all $h$
\begin{align*}
0 =& \langle \bar \partial^* \tau , \pi^* h \rangle \\
=& \langle \bar \partial^*(\zeta - \bar \partial(g)) + \bar \partial^* \gamma', \pi^*h \rangle \\
=& \langle \bar \partial^* \gamma', \pi^* h \rangle \\
=& \langle \gamma', \pi^* \bar \partial_{\bar B} h \rangle \\
=& \Vol(\cX_0) \langle \gamma'', \bar \partial_{\bar B} h \rangle_{\bar B}.
\end{align*}
Thus $\gamma''$ is closed and co-closed, hence vanishes by Lemma \ref{lem:PLball}. So $\tau = \zeta - \bar \partial (g)$, which in particular is in the closure of the image of $\bar \partial$, since $\zeta$ is. But the intersection of the closure of the image of $\bar \partial$ and $\mathcal{H}^{0,1}(\bar \cX)$ vanishes by Proposition \ref{prop:weakHodge}, so $\tau = 0$.
\end{proof}

Next, we extend the above to $(0,2)$-forms.
\begin{proof}[Proof of Lemma \ref{lem:vanishing02}]
Let $\alpha \in \mathcal{H}^{0,2}(\bar \cX)$. As in the proof of Lemma \ref{lem:vanishing01}, using that $H^{0,2}(\cX_b)=0$ for all $b\in B$, we obtain that $\alpha$ restricted to each fibre is exact, and so we can find a $(0,2)$-form $\eta$ on $\bar \cX$ which is in the closure of the image of $\bar \partial$ and such that the purely vertical parts of $\alpha$ and $\eta$ agree. Hence there are further $(0,1)$-forms $\tau_i$ and functions $h_{i,j}$ such that 
\begin{align*}
\alpha =& \eta + \sum_i \tau_i \wedge d\overline{z}_i + \sum_{i,j} h_{i,j} d\overline{z}_i \wedge d \overline{z}_j.
\end{align*}

When we now use $\bar \partial \alpha = 0$, we see that there is only one term which has two vertical components and one horizontal, namely using the purely vertical part of $\bar \partial \tau_i$ in 
$$\sum_i \bar \partial \tau_i \wedge d \overline{z}_i,$$
which then has to vanish. By the vanishing of the groups $H^{0,1}(\cX_b)$, each of the $\tau_i$ are therefore in the closure of the image of $\bar \partial$, up to a horizontal term. If we incorporate this term in the purely horizontal expression of $\alpha$ we see that we can write $\alpha$ as
\begin{align*}
\alpha =& \eta + \sum_i \zeta_i \wedge d\overline{z}_i + \sum_{i,j} h_{i,j} d\overline{z}_i \wedge d \overline{z}_j,
\end{align*}
where each $\zeta_i$ is in the closure of the image of $\bar \partial$. Now, as $\bar \partial (f) \wedge d\overline{z}_i = \bar \partial (f d \overline{z}_i)$ for a function $f$, we see that in fact $\zeta_i \wedge d\overline{z}_i$ is in the closure of the image of $\bar \partial$. So we can incorporate this term in $\eta$, and write $\alpha$ as
\begin{align*}
\alpha =& \eta + \sum_{i,j} h_{i,j} d\overline{z}_i \wedge d \overline{z}_j.
\end{align*}
Using closedness of $\alpha$, the functions $h_{i,j}$, and hence the $(0,2)$-form $$\gamma:=\sum_{i,j} h_{i,j} d\overline{z}_i \wedge d \overline{z}_j,$$ are pulled back from $B$. 

The rest of the argument then follows the proof of Lemma \ref{lem:vanishing01}. Introduce this time $\Phi$ on $\Om^{0,1}(\overline{B})$ by setting
$$
\Phi(\upsilon)= \langle \delb^* \eta, \pi^*\upsilon \rangle_{\bar \cX}.
$$
The unique continuous extention of $\Phi$ is represented as 
$$
\Phi(\upsilon) = \langle \tilde{\tau},\upsilon \rangle_{\bar B}
$$
for some $\tilde{\tau} \in L^2(\Omega^{0,1}(\bar B))$. Moreover, since $\cH^{0,1}(\bar B)$ vanishes by Lemma \ref{lem:vanishing01}, $\tilde \tau$ is in the image of the Laplacian on the base. Thus
$$
\Phi(\upsilon) = \langle \bar \partial \tau', \bar \partial \upsilon \rangle_{\bar B}+\langle \bar \partial^* \tau', \bar \partial^* \upsilon \rangle_{\bar B}
$$
for some $(0,1)$-form $\tau'$ on $\bar B$. In particular, for any $\upsilon\in \mathrm{Im}(\bar\partial^*)\subset L^2(\Omega^{0,1}(\bar B))$,
$$
\Phi(\upsilon)=\langle  \eta, \pi^*\delb \upsilon \rangle_{\bar \cX}=\langle \bar \partial \tau', \bar \partial \upsilon \rangle_{\bar B}.
$$
Then, replacing $\gamma$ by $\gamma'=\gamma+\frac{\pi^*\delb \tau'}{\Vol(\cX_0)}$, we can write
$$
\alpha=\eta-\delb\left(\frac{\pi^*\tau'}{\Vol(\cX_0)}\right)+\gamma',
$$
where $\gamma'$ is the pullback of a $\delb$-closed $(0,2)$-form $\gamma''$. Using that $\alpha$ is co-closed, together with the equality 
$$
\langle \bar \partial \tau', \bar \partial \upsilon \rangle_{\bar B} = \frac{1}{\Vol(\cX_0)} \langle \bar \partial \pi^*\tau', \bar \partial \pi^*\upsilon \rangle_{\bar \cX}
$$
we obtain
$$
\langle \gamma'',\delb \upsilon \rangle_{\bar B}=0
$$
for any $\upsilon\in \mathrm{Im}(\bar\partial^*)$. Taking $\upsilon=\bar\partial^*\gamma''$ yields that $\gamma''$ is a harmonic $(0,2)$-form, and so vanishes by Lemma \ref{lem:PLball}. Hence, $\gamma'=0$ and $\alpha$ is in the image of $\delb$, so has to vanish as well, being harmonic.
\end{proof}

Combining Lemma \ref{lem:vanishing02}, Proposition \ref{prop:weakHodge} and the discussion following it, we get the following
\begin{corollary}
\label{cor:dbar}
Suppose $H^{0,1}(\cX_0)=0$ and $H^{0,2}(\cX_0)=0$. Then for any smooth $(0,2)$-form $\beta$ on $\cX$ such that $\delb \beta = 0$ and any $\cX' \subset \cX$ with compact closure, there exists a smooth $(0,1)$-form $\tau$ such that $\delb \tau = \beta$ on $\cX'$.
\end{corollary}

\section{Reduction to a finite dimensional moment map problem}
We now show how to reduce the cscK equation on $(\cX_b, [\omega_{\uep}])$ to a finite dimensional moment map problem. This will be done in several stages. In Section \ref{sec:initial11form} we first use the results from the previous section to obtain an initial approximate solution on each fibre, which is globally $(1,1)$. In Section \ref{sec:fibreproblem} we then solve a fibrewise problem before we in Section \ref{sec:findimproblem} finally show that having done so, we are left with a moment map equation in finite dimensions, relying on the Dervan--Hallam point of view on the moment map interpretation of the cscK equation.

\subsection{Constructing the initial form globally}
\label{sec:initial11form}

We recall the previous setting.  Let $\cH^{1,1}(X_0,\R)$ stands for the space of $(\om,J_0)$-harmonic real $(1,1)$-forms on $(X_0, \om)$ and let $( \alpha_1, \ldots, \alpha_r)$ be a basis of this vector space. Note that by uniqueness of harmonic representatives of Dolbeault cohomology classes, the $\alpha_i$'s are $K$-invariant. For any $\uep:=(\ep_1,\ldots, \ep_r)$, we set 
$$
\om_{\uep}:=\om+\sum_{i=1}^r \ep_i \alpha_i.
$$
This defines a K\"ahler form on $X_0$ for all $\uep$ in a sufficiently small neighbourhood $U$ of zero in $\R^r$. For any $b\in Z$ and $\uep\in U$, we have $[\om_{\uep}]\in H^{1,1}(\cX_b,\R)$. However, $\om_{\uep}$ itself might not be of type $(1,1)$ with respect to $\Phi(b)$, and while the pullback (through $B\times M \to M$) of $\om_{\uep}$ is a globally closed form on $\cX$, it may not be $(1,1)$ there even if it is on the fibres. For this reason we perform a first perturbation of the form $\om_{\uep}$.
\begin{lemma}
\label{lem:first11form}
Up to shrinking $\cX$ and $B$, there exists a $(1,1)$-form $\omega'_{\uep}$ on $\cX$ such that $[\omega'_{\uep}]_b=[\omega_{\uep}]_b$ for all $b$ and $\omega'_{\uep} = pr_2^* \omega + O(|\uep|)$.
\end{lemma}
In the above $[*]_b$ denotes the corresponding second cohomology class on the fibre over $b \in B$.

If we pull back $\omega_{\uep}$ from the smooth projection $\cX \to M$, then we obtain a closed form $\widetilde \omega_{\uep}$ such that $[\widetilde \omega_{\uep}]_b=[\omega_{\uep}]_b$ for all $b \in B$. However, this is only $(1,1)$ on the central fibre of the family $\cX \to B$. It may not be $(1,1)$ on all of $\cX$, or even on the non-zero fibres. Note that using the assumption that $H^{0,2}(X_0)=0$, which implies $H^{0,2}(X_b)=0$ for all $b$ after potentially shrinking $B$, we could change $\widetilde \omega_{\uep}$ to become $(1,1)$ on each fibre of $\cX \to B$, using the Hodge decomposition on $(0,2)$-forms. However, this is not enough as $\widetilde \omega_{\uep}$ may still not be globally $(1,1)$. This is the reason we call upon the Hodge decomposition for complex manifolds with boundary from Section \ref{sec:cohomologyvanishing} which allow us to remove the $(0,2)$ and $(2,0)$-parts of $\omega_{\uep}$ globally on $\cX$. 

\begin{proof}[Proof of Lemma \ref{lem:first11form}]
Let $\alpha$ be some harmonic representative on $\cX_0$ as above. We will abuse notation and denote also by $\alpha$ the form $\pr^*_M\alpha$ on $\bar \cX$, where $\pr_M : \bar \cX \to M$ is the smooth projection to the second factor, using that as a smooth manifold $\bar \cX=\bar B \times M$. In what follows, we will be working on the total space $\cX$ (in particular $J$ stands for the almost-complex structure on $\cX$). Now the $(1,1)$-component of the $2$-form $\alpha$ is the fixed part of $\alpha$ under the map $\alpha \mapsto \alpha(J\cdot,J\cdot)$. Thus, its $(2,0)+(0,2)$-component is given by
 \begin{equation}
  \label{eq:02part}
 \alpha^{(2,0)+(0,2)}=\frac{\alpha-\alpha(J\cdot, J\cdot)}{2}.
 \end{equation}
 In particular, for fixed $\alpha$, we see that this is smooth. 
  
Now, consider the forms $\alpha_i$. They are $d$-closed, and thus satisfy $\delb \alpha_i^{(0,2)}=0$. Moreover, they satisfy the $\bar \partial$-Neumann conditions. Indeed, since the forms are $\bar \partial$-closed, the condition $\sigma_{\bar \partial \alpha_i^{(0,2)}} = 0$ is trivially satisfied. Moreover, as $\alpha_i$ is pulled back from $M$, $(dR^{0,1} \wedge \beta_p, \alpha_i) = 0$ for all $(0,1)$-forms $\beta_p$ at $p \in \bar \cX$, i.e. $\sigma_{\alpha_i} = 0$. But the decomposition of $\alpha_i$ into its $(2,0)$, $(1,1)$ and $(0,2)$-components is an orthogonal decomposition at every point. So $\sigma_{\alpha^{0,2}_i} = 0$, too.

From Corollary \ref{cor:dbar} we deduce that there is $\theta_i'\in \Om^{(0,1)}(\cX,\R)$ such that $\delb \theta_i'=\alpha_i^{(0,2)}$, where we have now shrunk $\cX$ slightly to ensure that we have a smooth such form (note that the shrinking is independent of $i$). Let $\theta_i'' = \overline{\theta_i'}$, and set $\theta_i=\theta_i'+\theta_i''$, which is a real $(2,0)+(0,2)$-form. Since the $\alpha_i$ are real $\alpha_i^{2,0}=\overline{\alpha_i^{0,2}}$, and so we have that $\partial \theta_i'' = \overline{\bar \partial \theta_i'}=\alpha_i^{(2,0)}$. Let
\begin{align*}
 \om'_{\uep} & = \displaystyle\om+\sum_{i=1}^r \ep_i(\alpha_i - d \theta_i) \\
              & =  \displaystyle\om_{\uep}-d \left( \sum_{i=1}^r\ep_i \theta_i\right)\\
              & = \om_{\uep}- d \theta_{\uep} 
 \end{align*}
 where 
 $$
 \theta_{\uep}=\sum_{i=1}^r\ep_i \theta_i.
 $$
Then $\om'_{\uep}$ is globally $(1,1)$ on $\cX$. Since we have changed $\om_{\uep}$ by a globally closed form, its induced class on each fibre remains the same, and as the $\theta_i$ are independent of $\uep$, it is clearly $O(|\uep|)$, as required.
\end{proof}

\subsection{Solving the fibrewise problem}
\label{sec:fibreproblem}
From now on, we assume that $(M, J_0,\om)$ is cscK.
At this point we have, for each $\uep \in U$, a holomorphic submersion 
$$
(\cX, \om'_{\uep}) \to B 
$$
where $\omega'_{\uep}$ is relatively K\"ahler and whose induced class on each fibre is $[\om_{\uep}]$. Furthermore, we have a moment map $\nu_{\uep} : \cX \to \k^*$.  Let
$$
\k_{\uep} = \{ \nu_{\uep}(v) : v \in \k \}
$$
be the space of potentials with respect to $\om'_{\uep}$, and more generally, for a function $\phi : \cX \to \R$, let 
$$
\k^{\phi}_{\uep} = \{ \nu_{\uep}(v) + v(\phi) : v \in \k \}
$$
be the space of potentials with respect to $\om'_{\uep}+ \ddb \phi$. Further, for a function $\phi_b : \cX_b \to \R$ on the fibre, define 
$$
\k^{\phi_b}_{\uep,b} = \{ \nu_{\uep}(v)\big|_{\cX_{b}} + \frac{1}{2} \langle v, \nabla_{\omega'_{\uep,b}} (\phi_b) \rangle_{\om'_{\uep,b}} : v \in \k \} \subset C^{\infty}(\cX_b).
$$
Note that $\k^{0}_{\uep,b} = \k_{\uep}\big|_{\cX_{b}}$, i.e. the space of functions obtained by restricting the functions in $\k_{\uep}$ to $\cX_{b}$. However, if $\phi : \cX \to \R$ is a function and we let $\phi_{b} = \phi\big|_{\cX_{b}}$, then 
$$
\k^{\phi_b}_{\uep,b} \neq \k^{\phi}_{\uep}\big|_{\cX_{b}}
$$
in general. The two only agree if $b$ is a fixed point of the $K$-action on $B$.

Next, we solve a fibrewise problem to improve the approximate solution and ultimately reduce the cscK equation to a finite dimensional problem. Let $J_{b}$ be the almost-complex structure on $M$ such that $\cX_b = (M, J_b)$.
\begin{proposition}
\label{prop:approxsoln}
 Up to shrinking $U$ and $B$, there exists a smooth map 
 $$
\phi : U \times B \times M \to \R
 $$
such that  for all $(\uep, b) \in U\times B$
$$
\scal(M, J_{b}, \om'_{\uep,b}+\ddb\phi_{\uep,b})\in \k^{\phi_{\uep,b}}_{\uep,b},
$$
where $\phi_{\uep,b}(x) = \phi(\uep, b, x)$ for $x \in M$.
\end{proposition} 
\begin{proof}
Let $W$ be an open subset of $C^{k+4,\alpha}(M)$ whose elements are K\"ahler potentials of regularity $C^{k+4,\alpha}$ for each $(\uep,b) \in U \times B$. Consider the map 
$$
\Phi : U \times B \times W \times \mathfrak{k} \times \R \to C^{k,\alpha}(M)
$$
given by 
$$
(\uep, b, \phi, v, c) \mapsto \scal(M, J_b, \omega'_{\uep} + \ddb_b(\phi)) - (\nu_{\uep}(v) + v(\phi)+c)_b.
$$
We seek to show that there is a zero of the above map for every sufficiently small $\uep \in U$ and $b \in B$ as $\nu_{\uep}(v) + v(\phi) \in \mathfrak{k}^{\phi}_{\uep,b}$.

For fixed $(\uep,b)$, the linearisation of the equation at $(0,0) \in W \times \mathfrak{k}$ is the map 
$$
\Psi_{\uep,b} : C^{k+4,\alpha}(M) \times \mathfrak{k} \times \R \to C^{k,\alpha}(M)
$$
given by
$$
(\phi, v, c) \mapsto L_{\uep,b}(\phi) - (\nu_{\uep}(v) + c)_b,
$$
where $L_{\uep,b}$ is the linearisation of the scalar curvature at $(J_b, \omega'_{\uep})$ when varying the K\"ahler form. This varies smoothly with $(\uep,b)$ and moreover when these two parameters are $0$, we have that $L_{0,0}$ is the negative of the Lichnerowicz operator of $(J_0, \omega)$, since this is $\scal(M,J_0,\omega)$ is constant. The Lichnerowicz operator is a self-adjoint operator whose kernel, and therefore co-kernel, consists precisely of the functions $\nu_{0}(v)+c$ for $v \in \mathfrak{k}$ and $c \in \R$. In particular, $\Psi_{0,0}$ is surjective and admits a right-inverse $Q_{0,0}$, e.g. by choosing the unique representative in $C^{k+4,\alpha}(M)$ which is orthogonal to all the $\nu(v) + c$. 

It follows that for all sufficiently small $(\uep,b)$, the operators $\Psi_{\uep,b}$ are surjective and admits right-inverses $Q_{\uep,b}$ whose operator norm satisfies 
$$
\frac{1}{2}\| Q_{0,0} \| \leq  \| Q_{\uep,b} \| \leq 2 \| Q_{0,0} \|.
$$
Moreover, as the operators $\Phi_{\uep,b}$ vary smoothly we can ensure that the $Q_{\uep,b}$ also do. Now, as $\omega'_{\uep,b} = \omega + O(|\uep|,|b|)$, we can shrink $U$ and $B$ such that the uniform estimate 
$$
\| \scal(M,J_b,\omega'_{\uep,b}) - \scal(M,J_0,\omega) \|_{C^{k+4, \alpha}} \leq C
$$
holds, for any desired constant $C>0$. Note also that $\scal(M,J_0,\omega)$ in fact is a constant. In particular, by picking $C$ sufficiently small, this allows us to apply the Quantitative Implicit Function theorem to deduce that there is a root of $\Phi$ for all such $(\uep,b)$, which varies smoothly with $(\uep,b)$ since $Q_{\uep,b}$ does. Elliptic regularity then implies that these are also smooth in $M$. This is what we wanted to show.
\end{proof}

\subsection{The finite dimensional problem}
\label{sec:findimproblem}
As the map $\phi$ is smooth, we can for each $\uep \in U$ define a smooth map $\phi_{\uep} : \cX \to \R$ whose restriction to $\cX_{b}$ is $\phi_{\uep,b}$, using that $\cX = B \times M$ as a smooth manifold. From now on we will let $\om_{\uep} = \omega'_{\uep} + \ddb \phi_{\uep}$, where we are relabelling $\om_{\uep}$ from Section \ref{sec:initial11form} to reduce the notation.

By Dervan--Hallam, we can for each $\uep \in U$ obtain a moment map $\sigma_{\uep}$ on $B$ in the following way. In general, let $p : Y \to V$ be a holomorphic submersion, with compact fibres of relative dimension $n$. The base $V$ does not need to be compact. Suppose a Lie group $K$ acts on $Y$ and $V$ making $p$ equivariant. Suppose further that there is a closed $(1,1)$-form $\omega_Y$ on $Y$ which is relatively K\"ahler and which admits a moment map $\nu_Y : Y \to \mathfrak{k}^*$. Let $\hat S_v$ be the average scalar curvature of the fibre $Y_v$ and let $\rho$ denote the relative Ricci form, the curvature of the Hermitian metric induced by $\omega_Y$ on the top exterior power of the vertical tangent bundle $\ker p$. Let $\langle \cdot , \cdot \rangle : \mathfrak{k}^* \times \mathfrak{k} \to \R$ denote the natural pairing of a dual vector and a vector.
\begin{theorem}[{\cite[Theorem 4.6]{dervanhallam23}}]
\label{thm:dervanhallam}
In the setting above, the map $\mu_V : V \to \mathfrak{k}^*$ given by
\begin{align}
\label{eq:futinv}
\langle \mu_V(v), \xi \rangle = \int_{Y_v} \langle \nu_Y, \xi \rangle (S(\omega_Y\big|_{Y_v}) - \hat S_v) \omega_Y^n
\end{align}
is a moment map for the $K$-action on $V$ with respect to the closed $(1,1)$-form
$$
\Omega = \frac{\hat S_v}{n+1} \int_{Y/V} \omega_Y^{n+1} - \int_{Y/V} \rho \wedge \omega_Y^{n}.
$$ 
\end{theorem}
The form $\Omega$ is the Weil--Petersson form on the base $V$. Note that in this generality it does not need to be positive. 

Now, if we apply Theorem \ref{thm:dervanhallam} to the metrics $\omega_{\uep}$ constructed in Proposition \ref{prop:approxsoln}, we then obtain moment maps $\mu_{\uep}$ with respect to the corresponding Weil--Petersson forms $\Omega_{\uep}$. In our case this form is positive. At $\uep = 0$, we have the usual Weil--Petersson form $\Omega$ on the Kuranishi space, which is a positive form. Since $\omega_{\uep} = \omega + O(|\epsilon|)$, it follows that the corresponding $\Omega_{\uep}$ satisfy this, too. Thus, the $\Om_{\uep}$ are positive for all sufficiently small $\uep$. So for each $\uep \in U$, we have a K\"ahler metric $\Omega_{\uep}$ on $B$ and a moment map for the $K$-action on $B$, which we will call $\mu_{\uep}$. 

Next, we claim that a zero of the moment map $\mu_{\uep}$ corresponds to a cscK metric on the fibre. Indeed, a zero of the moment map means the scalar curvature is orthogonal to the space $\k^{\phi_{\uep}}_{\uep}\big|_{\cX_{b}}$. If it were the case that this equals $\k^{\phi_{\uep,b}}_{\uep,b}$, then a zero of the moment map would clearly correspond to a cscK metric on the fibre, since the scalar curvature lies in $\k^{\phi_b}_{\uep,b}$. As remarked above this may not be true. However, a projection argument from \cite[Lemma 4.15]{ortusektnan24}, shows that as $\k^{0}_{\uep,b} = \k_{\uep}\big|_{\cX_{b}}$, the $L^2$-orthogonal projection $\k^{\phi_b}_{\uep,b} \to \k^{\phi}_{\uep}\big|_{\cX_{b}}$ with respect to $\omega_{\uep}$ is an isomorphism for all sufficiently small $\uep$, and so for such $\uep$, a zero of the moment map still has the same interpretation. 

In summary, we get the following, which reduces the existence of a cscK metric to a finite dimensional moment map problem.
\begin{lemma}
\label{lem:zeroofmomentmap}
 Up to shrinking $U$, we have that for all $\uep \in U$, the form $\Omega_{\uep}$ on $B$ is K\"ahler. Moreover, the corresponding moment map $\mu_{\uep}$ for the $K$-action on $B$ satisfies that $\mu_{\uep}(b) = 0$ if and only if $\scal(M, J_{b}, \om_{\uep,b})$ is constant. 
\end{lemma}

\section{Solving the finite dimensional problem : abelian case}
\label{sec:finitedimproblem}

We thus reduced the problem to finding zeros of $\mu_{\uep}$ on $B$. Our goal now is to relate the vanishing of $\mu_{\uep}$ to $K$-polystability of some fiber $(\cX_b,[\om_{\uep}])$, for some fixed $b\in B$. In this whole section, we assume that $K=T$ is a torus.

\subsection{Focusing on an orbit}
\label{sec:focusingorbit}
We start by restricting our study to the (local) $T^\C$-orbit of $b$ in $B$. By the Hilbert--Mumford criterion, there is a one parameter subgroup $\rho : \C^* \to T^\C$ such that
$$
\lim_{t\mapsto 0} \rho(t)\cdot b =b_0\in B
$$
and $b_0$ has a closed $T^\C$-orbit (this follows more precisely by a combination of \cite[Section 8.3]{Humphreys} that shows that the orbit closure of $b$ contains a closed orbit and \cite[Theorem 1.4]{Kempf} that implies that this closed orbit can be reached by a one-parameter subgroup). By \cite{gabor-deformations,ortu24}, $(\cX_{b_0}, [\om])$ admits a cscK metric. Moreover, $(\cX_b,[\om])$ is a small complex polarised deformation of $(\cX_{b_0}, [\om])$. Note that $h^{0,j}(\cX_{b_0})=h^{0,j}(\cX_0)$, and that by \cite[Corollary 3.2]{dervan-moduli}, $\Aut_0(\cX_{b_0}, [\om])\simeq T_{b_0}$. In particular,  $T_{b_0}$ is abelian. Hence, $(\cX_{b_0}, [\om])$ satisfies the same hypothesis as $(\cX_0, [\om])$ so we might, and we will, up to performing Kuranishi--Fujiki--Schumacher--Sz\'ekelyhidi's slice starting from $(\cX_{b_0}, [\om])$, assume that $b_0=0$. We will keep the notations $T$, $\mu_{\uep}$, etc for the group $T_{b_0}$, the associated moments maps, etc. 

We now reduce to the case where the stabiliser $T_b\subset T$ of $b\in B$ is trivial. Denote by $\t_b$ the Lie algebra of $T_b$.
 \begin{lemma}
  \label{lem:hamiltonianstabiliservanish}
  Let $(\uep,b)\in U\times B$ such that $(\cX_b, [\om_{\uep}])$ is $K$-polystable. Then, for all $b'\in(T^\C\cdot b )\cap B$, and for all $u\in \t_{b'}$, one has $\langle \mu_{\uep}(b'),u\rangle=0$.
 \end{lemma}
\begin{proof}
If $b'\in(T^\C\cdot b )\cap B$, then $\cX_b$ and $\cX_{b'}$ are isomorphic, so $\cX_b$ is $K$-polystable if and only if $\cX_{b'}$ is. Thus it is enough to show that for all $u\in \t_b$, the hamiltonian $\langle \mu_{\uep},u\rangle=0$  vanishes at $b$. But by construction
\begin{equation}
 \label{eq:futaki stabiliser}
\langle \mu_{\uep}(b),u\rangle=\int_{\cX_b} \langle \nu_{\uep}, u \rangle (S(\omega_{\uep,b}) - \hat S_b) \omega_{\uep}^n.
\end{equation}
As $b$ is a fixed point under the $T_b$-action, for any $u\in\t_b$ this
is the classical Futaki invariant of the holomorphic vector field associated to $u$, which therefore vanishes by the polystability assumption.
\end{proof}
 As elements in $T_b$ fix all the points in $B\cap \overline{T^\C}\cdot b$, and as by Lemma \ref{lem:hamiltonianstabiliservanish} the hamiltonians $\langle \mu_{\uep},u\rangle$ for $u\in \t_b=\Lie(T_b)$ vanish if we assume $K$-polystability, elements in $\t_b$ will play no role in the arguments that follow. Hence, working instead with the torus $T/T_b$, the maps $\mu_{\uep}$ take values in $(\t/\t_b)^*$, so we may, and we will, assume that $T_b$ is trivial.

 We now set 
 $$V:= H^{0,1}_\om(T^{1,0}),$$
 the first cohomology group as defined in Section \ref{sec:deformationsandslice}, associated to the complex structure of $\cX_0$. There is a weight decomposition under the $T$-action
 \begin{equation}
  \label{eq:weightdecomposition}
 V:= \bigoplus_{m\in \M} V_m
 \end{equation}
 for $\M\subset \t^*$ the lattice of characters of $T$. This means that if we decompose $b\in V$ accordingly as
 $$
 b=\sum_i b_{m_i},
 $$
then the action of a one-parameter subgroup generated by $v\in \t$ is given by
\begin{equation}
 \label{eq:explicitTaction}
t\in\C^* \mapsto \sum_i t^{\langle m_i, v \rangle} b_{m_i},
\end{equation}
where $\langle \cdot, \cdot \rangle$ is the natural duality pairing on $\t^*\times \t$.
Finally, we will restrict our setting to
 $$
 B\cap \bigoplus_{b_{m_i\neq 0}} V_{m_i},
 $$
which we still denote by $B$. That is, we only consider weight spaces actually appearing in the decomposition of $b$.

\subsection{A local rigidity result for families of moment maps}
\label{sec:localrigidity}
We next state a local rigidity result for the images of a complexified torus orbit under a continuous family of moment maps. Here is the setting of last section. We have :
\begin{enumerate}
 \item[($R_1$)] An effective and holomorphic action of a compact torus $T$ on a complex vector space $V$;
 \item[($R_2$)] A continuous family of symplectic forms $(\Om_{\uep})_{\uep\in U}$ on a ball $B\subset V$ around the origin, with respect to which the $T$-action is hamiltonian;
 \item[($R_3$)] A point $b\in B$ with trivial stabiliser, $0$ in its $T^\C$-orbit closure, and such that for all weight $m\in \M$ appearing in the weight space decomposition of $V$, $b_m\neq 0$.
 \item[($R_4$)]  The restriction of the symplectic form $\Om_0$ to the $T^\C$-orbit of $b$ is non-degenerate.
\end{enumerate}
\begin{remark}
Note that in what follows, we will not use the fact that the forms $\Om_{\uep}$ are K\"ahler with respect to the complex structure on $V$. The only hypothesis we need instead is the weaker $(R_4)$. 
\end{remark}
We still denote the equivariant moment mappings $\mu_{\uep} : B \to \t^*$, and we set $$\mathcal{Z}:=B\cap(T^\C\cdot b)$$ and $$\ocZ:=B\cap(\overline{T^\C\cdot b}).$$
We also introduce 
$$
\sigma:= \sum_{b_{m_i}\neq 0} \R_+ \cdot m_i\subset \t^*
$$
with $\lbrace m_i \rbrace$ the weights of the decomposition (\ref{eq:weightdecomposition}) under the torus action,
and for $\eta >0$ 
$$
\sigma_\eta := \sum_{b_{m_i}\neq 0} [0,\eta ) \cdot m_i\subset \t^* .
$$
Then, under the assumptions $(R_1)-(R_4)$, we have :
 \begin{proposition}
 \label{prop:sigma-eta-image-orbit}
  Up to shrinking $B$ and $U$, there exists $\eta>0$ such that for all $\uep\in U$,
  $$
  \mu_{\uep}(0)+ \sigma_\eta  \subset \mu_{\uep}(\ocZ).
  $$
 \end{proposition}
Note that by the local version of Atiyah and Guillemin--Sternberg convexity theorem, combining the equivariant Darboux Theorem (\cite[Theorem 3.2]{Dwivedi}) and the local description of linear hamiltonian torus actions (\cite[Section 7.1]{Dwivedi}), there are $\eta_{\uep}>0$ and open neighbourhoods $0\in B_{\uep}\subset B$ such that for all $\uep$,
 $$
 \mu_{\uep}(B_{\uep})= \mu_{\uep}(0)+ \sum_i [0,\eta_{\uep}) \cdot m_i .
 $$
 By continuity of $\uep\mapsto \mu_{\uep}$, we can find a uniform $\eta$ (up to shrinking $B$ and $U$) so that for all $\uep\in U$,
 $$
 \mu_{\uep}(0)+ \sum_i [0,\eta) \cdot m_i  \subset \mu_{\uep}(B).
 $$
 Hence, the content of Proposition \ref{prop:sigma-eta-image-orbit} is to preserve this inclusion when restricting to $\ocZ$. As the proof  is quite technical, we postpone it to Appendix \ref{sec:appendix}.
 \begin{remark}
\label{rem:sjamaar}
In a first version of this work, a different strategy was used to produce symplectic forms $\Om_{\uep}'$ on $B$. In particular, those forms were not compatible with the complex structure on $V$. Hence we couldn't use the results in \cite[Sections 4 and 6]{Sjamaar} directly. This is why instead we use a more direct argument, that requires an explicit knowledge of the $T$-action, and this is the main reason for our hypothesis $K=T$. We expect however that a generalisation of Proposition \ref{prop:sigma-eta-image-orbit} should hold for general compact connected groups $K$, considering the intersection of the images of the moment maps with a fixed positive Weyl chamber, and replacing $\sigma$ with an appropriate cone, as in \cite[Theorem 6.5]{Sjamaar}. With our new approach, as now the forms $\Om_{\uep}$ are compatible with the complex structure on $V$, Sjamaar's results could be used to relax our asumption on $K$ being abelian. We will instead provide a different proof of this fact in Section \ref{sec:solve general}.
 \end{remark}
 We will also need the following lemma.
 \begin{lemma}
  \label{lem:sigmaproperties}
  The cone $\sigma\subset \t^*$ is a strongly convex rational polyhedral cone of maximal dimension $\dim(\t^*)$. Hence, so is its dual cone $\sigma^\vee$.
 \end{lemma}
 Recall that by definition, 
 $$
 \sigma^\vee=\lbrace v\in \t\,\vert\:\forall m\in \sigma,\; \langle m , v \rangle \geq 0 \rbrace.
 $$
\begin{proof}
 The cone $\sigma$ is a rational polyhedral cone by definition. As $0\in \overline{T^\C\cdot b}$, and as the action of $t\in T^\C$ on $b$ is given by
 \begin{equation}
  \label{eq:action}
 \sum_i t^{\langle m_i, \cdot \rangle}b_{m_i},
 \end{equation}
 we know that there is $v\in \t$ such that the action of $v$ sends $b$ to $0$. That is, there is $v\in \t$ such that for all $i$,
 $$
 \langle m_i, v \rangle >0.
 $$
 But then $\sigma$ cannot contain any line, as any point $x\in \sigma\cap (-\sigma)$ would give $\langle x , v \rangle >0$ and $\langle -x, v \rangle >0$, which is absurd. Finally, $\dim(\sigma) = \dim(\t^*)$ follows from the fact that any element $v\in\sigma^\perp$ belongs to the stabiliser of $b$ by the description of the $T^\C$-action at $b\in B$. As $b$ is supposed to have trivial stabiliser, we deduce that $\dim(\sigma^\perp)=0$, and thus $\dim(\sigma)=\dim(\t^*)$ by duality. The statement about $\sigma^\vee$ follows from general theory of polyhedral cones, see e.g. \cite[Appendix]{Oda}.
\end{proof}
The previous lemma has several interesting consequences. First, there is a finite set of rational elements $\lbrace v_1,\ldots, v_\ell\rbrace\in \t^\ell$ such that
$$
\sigma^\vee = \sum_j \R_+\cdot v_j.
$$
Then, the facets of $\sigma$ (resp. $\sigma^\vee$) are of the form $\sigma\cap H_{v_i}$ (resp. $\sigma^\vee\cap H_{m_i} $), for  some of the $m_i$'s (resp. $v_i$'s), where we set $H_v=\lbrace m\in\t^*, \: \langle m , v\rangle = 0 \rbrace$ (resp. $H_m=\lbrace v\in\t, \: \langle m , v\rangle = 0 \rbrace$). In general, a face of $\sigma^\vee$ is of the form 
$$
(\cap_i H_{m_i})\cap \sigma^\vee
$$
for some $m_i$'s. Finally, note by construction that the rays $\rho\in\sigma(1)$ (that is $1$-dimensional faces of $\sigma$) are all of the form $\R_+\cdot m_i$, for the $m_i$'s such that $\sigma^\vee\cap H_{m_i}$ is a facet of $\sigma^\vee$.

\subsection{Solving the problem}
\label{sec:solvingproblem}
We can now conclude the proof of Theorem \ref{thm:main}, in the abelian case. We restrict to variations in a single direction in the K\"ahler cone. Let 
$$
\alpha = \sum_i \lambda_i \alpha_i\in \cH^{1,1}(X_0,\R).
$$
We set $\om_\ep=\om+\ep \alpha$. There is $\ep_0 >0$ such that for $\ep\in (-\ep_0, \ep_0)$, $(\ep \lambda_i)_{1\leq i\leq r}\in U$. We denote by a subscript $\ep$ all quantities depending on $\uep := (\ep \lambda_i)_{1\leq i\leq r}$.  Our goal is to prove that up to shrinking $\ep_0$, for all $\vert \ep\vert < \ep_0$, if $(\cX_b, [\om_\ep])$ is $K$-polystable, then it carries a cscK metric. 

From the previous sections, to produce a cscK metric, it is enough to find $b'\in \ocZ$ such that $\mu_\ep(b')=0$. Indeed, if such a $b'$ is in $\cZ$, then $(M, J_{b'}, \om_{\ep,b})$ is cscK (cf Lemma \ref{lem:zeroofmomentmap}) and isomorphic to $\cX_b$ by the properties of the slice in Proposition \ref{prop:firstslice}. If $b'$ is in the boundary of $\ocZ$, then $(\cX_b, [\om_\ep])$ is a small polarised deformation of $(\cX_{b'}, [\om_\ep])$. As the latter is cscK, if the former is $K$-polystable, we can conclude from \cite{gabor-deformations,ortu24} that $(\cX_b, [\om_\ep])$ carries a cscK metric.

We fix $\eta$ as in Proposition \ref{prop:sigma-eta-image-orbit}. Then, it is is enough to prove that $K$-polystability implies $0\in \mu_\ep(0)+\sigma_\eta$ to complete the proof. But this is equivalent to 
$$
- \mu_\ep(0)\in \sigma_\eta.
$$
Up to shrinking $\ep_0$, and as $\mu_0(0)=0$, we only need to prove that $K$-polystability implies
$$
- \mu_\ep(0)\in \sigma.
$$
By the duality Theorem (\cite[Appendix, Theorem A.1]{Oda}), we have 
$
\sigma^{\vee\vee}=\sigma,
$
where $\sigma^\vee= \sum_j \R_+\cdot v_j$ is the dual of $\sigma$, for some finite set of rational elements $\lbrace v_1,\ldots, v_\ell\rbrace\in \t^\ell$. Hence, we need to show that $K$-polystability implies 
$$
-\langle \mu_\ep(0), v_j \rangle \geq 0,
$$
for all $j\in\lbrace 1,\ldots,\ell\rbrace.$ But any such $v_j$ lies in $\sigma^\vee$, and thus satisfies $\langle m_i, v_j \rangle \geq 0$ for all $m_i$ such that $b_{m_i}\neq 0$. By Equation (\ref{eq:explicitTaction}), each $v_j$ then generates a $1$-parameter subgroup $\rho_j : \C^* \to T^\C$ with 
$$
\lim_{t\mapsto 0} \rho_j(t)\cdot b = b_j \in \ocZ.
$$

From \cite[proof of Theorem 2]{gabor-deformations}, one can produce out of $\rho_j$ a regular test configuration $(\cX_j,\cA_{j,\ep})$ for $(\cX_b, [\om_\ep])$ with Futaki invariant given by (cf Equation (\ref{eq:futaki stabiliser})
$$
\mathrm{DF}(\cX_j,\cA_{j,\ep})=-\langle \mu_\ep(b_j), v_j \rangle.
$$
If $(\cX_b, [\om_\ep])$ is $K$-polystable, then we must have 
$$
\mathrm{DF}(\cX_j,\cA_{j,\ep}) \geq 0.
$$
 Moreover, the quantity $-\langle \mu_\ep(b_j), v_j \rangle$ is the Futaki invariant of $v_j$ that belongs to the stabiliser $\t_{b_j}$. Hence, from the independence of the Futaki invariant on the metric in a fixed K\"ahler class (\cite{futaki}), the map $b'\mapsto \langle \mu_\ep(b'), v_j \rangle$ is invariant along the $T^\C$-orbit of $b_j$. Note that $0$ belongs to the $T^\C$-orbit closure of $b_j$ (we can use the same $\C^*$-action that sends $b$ to $0$ to send any $b_j$ to $0$, as shown by Equation (\ref{eq:explicitTaction})). Thus, from these two observations, we see that
 $$
 -\langle \mu_\ep(0), v_j \rangle=\mathrm{DF}(\cX_j,\cA_{j,\ep})\geq 0.
 $$
which concludes the proof.
 
 \subsection{The Futaki invariant and local wall-crossing}
 \label{sec:invariantsandwallcrossing}
 In this section, we describe the local wall-crossing phenomena alluded to in the introduction. First, from the previous proof, it follows that for $\ep$ small enough, $(\cX_b, [\om_\ep])$ is $K$-polystable if and only if $\mathrm{Fut}_{[\om_\ep]}(\cX_b,\cdot )=0$ and for $j \in \lbrace 1,\ldots, \ell \rbrace$, one has $\mathrm{Fut}_{[\om_\ep]}(X_0, v_j)>0$, where $\mathrm{Fut}_{[\om_\ep]}(X ,\cdot)$ denotes the classical Futaki invariant of $(X, [\om_\ep])$. By invariance and continuity of the Futaki invariant, this is equivalent to 
 \begin{equation}
  \label{eq:finiteconditions}
\left\{ 
\begin{array}{ccccc}
\mathrm{Fut}_{[\om_\ep]}(X_0, v) & = & 0 & \mathrm{ for } & v \in \t_b, \\
\mathrm{Fut}_{[\om_\ep]}(X_0, v) & > & 0 &\mathrm{ for } & v \in \sigma^\vee.
\end{array}
\right.
 \end{equation}
While $K$-polystability immediately implies those conditions, the converse follows as those conditions implies the existence of a cscK metric from our proof of Theorem \ref{thm:main}, which then ensures $K$-polystability (\cite{stoppa,dervan-relative}). 

There are several interesting consequences from such a characterisation of $K$-polystability. First, observe that conditions $(R_1)-(R_4)$ required for Proposition \ref{prop:sigma-eta-image-orbit} are satisfied for {\it any} point $b'=\sum_{m_i} b'_{m_i}$ with same non-vanishing coordinates $b'_{m_i}$ as $b$ in the weight space decomposition. Hence, the result of Proposition \ref{prop:sigma-eta-image-orbit} holds true uniformly in $b$ for any compact family of such points. Since $T$ is a torus, the stabilisers $\t_{b'}$ of those points are all equal to $\t_b$, as they are simply 
$$(\mathrm{Span}_\R\lbrace m_i, \; b_{m_i}\neq 0 \rbrace)^\perp\subset \t,$$
and our proof works uniformly in $b$ for $b$ varying in a compact set $S$ of the above kind. More precisely, for such a compact family of polarised manifolds $(\cX_b, [\om_\ep])_{b\in S }$, there is $\ep_0 >0$ such that for $\vert \ep \vert  < \ep_0$, $K$-polystability of $(\cX_b, [\om_\ep])$ implies the existence of a cscK metric and is equivalent to the  System (\ref{eq:finiteconditions}). But conditions (\ref{eq:finiteconditions}) are independent of $b$ in this compact family. Hence, we deduce that if {\it one} of the $(\cX_b, [\om_\ep])$ is $K$-polystable for $b\in S$, then {\it all} of them are. For $\t_b=\lbrace 0 \rbrace$, this is simply openness of $K$-stability, and is implied by openness of the existence of a cscK metric when the reduced automorphism group is trivial.

 \section{Solving the finite dimensional problem : general case}
 \label{sec:solve general}
We now show how to deduce the existence of a cscK metric from K-polystability in general, using the approach of Ortu. The argument of Ortu in \cite[Theorem A.1]{ortu24} deals with the case $\uep = 0$. The argument relies on the moment map flow. For simplicity we first present the argument for the case when $b$ has discrete stabiliser. See Remark \ref{rem:nontrivstab} for details on how to adapt the argument in the general case. 

The moment map flow is the flow on $B$ given by
$$
\frac{d}{dt} b_t  = J(\xi(\mu_{\uep}(b_t)))
$$
with initial condition $b_0=b$, where $\xi(\mu_{\uep}(b_t))$  is the infinitesimal vector field on $B$ associated to $\mu_{\uep}(b_t)$. Crucially, after potentially shrinking $B$, there is a neighbourhood about the origin in $B$ such that the moment flow stays in this neighbourhood for all time. This follows if the origin is in the orbit closure of the starting point $b$. For our $b$ of interest this may not be true for the initial Kuranishi family about $\cX_0$. However, after an application of the Luna slice theorem \cite{luna} we may replace $\cX_0$ by another central fibre $\cX_0'$ which is cscK with respect to the initial polarisation $\omega$ such that $\cX_b$ lies in the Kuranishi family of $\cX_0'$ and $0$ now is in the orbit closure of $b$ on the corresponding base family, which we still denote $B$. For the flow at $\uep = 0$, there is a ball about the origin such that the flow is inward pointing on the boundary of this ball, with a uniform lower bound for the length of the inward pointing tangent vector (see \cite[Proposition 4.5]{DMS}). Such a bound persists for all $\uep$ sufficiently close to the origin, since the K\"ahler form and moment map on $B$ then is a small perturbation of that at $\uep = 0$, ensuring that the flow stays in the same ball for all such $\uep$. This in particular implies convergence of the flow to some $b_{\infty} \in B$, see  \cite[Proposition 2.4]{ortu24}. 

From \cite[Corollary 4.14]{DMS} and \cite[Proposition 2.6]{ortu24}, either 
\begin{itemize}
\item the limit point is a zero of the moment map and is in the $K^{\mathbb{C}}$-orbit of $b$, or;
\item the limit point does not lie in the $K^{\mathbb{C}}$-orbit of $b$, and there exist a $b' \in B$ and $v \in \k$ such that  
$$
\lim_{t \to \infty} \exp(-itv.b) = b'
$$
and $\langle \mu_{\uep}(b'), v\rangle \geq 0$.
\end{itemize}

In the first case, we produce a cscK metric, by Lemma \ref{lem:zeroofmomentmap}. We therefore want to use K-polystability to rule out the second case. This is done by identifying $\langle \mu_{\uep}(b'), v\rangle$ with the Donaldson-Futaki invariant of a test configuration in the following way. The holomorphic vector field induced by $v$ determines a non-trivial test configuration $(\cW, [\alpha_{\uep}])$ for $(\cX_b, [\omega_{\uep,b}])$, with smooth central fibre $(\cX_{b'}, [\omega_{\uep,b'}])$. Its Donaldson-Futaki invariant is therefore given by the negative of the Futaki invariant of the central fibre. But the Futaki invariant of $v$ on $\cX_{b'}$ is precisely given by the formula \eqref{eq:futinv} for the value of the moment map $\mu_{\epsilon}$ at $b'$ paired with $v$. K-polystability of $\cX_b$ with respect to $[\omega_{\uep,b}]$ implies that this has to be negative, which rules out the second case in the dichotomy above.

\begin{remark}
\label{rem:nontrivstab}
In the case when $b$ has non-discrete stabiliser, one has to adapt the argument by considering the moment map flow with respect to a complementary torus $T^{\perp}$ to a maximal torus in the stabiliser of $b$. The flow then converges to a zero of the moment with respect to this complementary torus, and further this implies that there is a critical point $b'$ of the original moment map in the orbit of the limiting point. Using that the limiting fibre has to be isomorphic to that of $(\cX_b, [\omega_{\uep,b}])$ since this is K-polystable, Ortu showed that $b'$ is actually a zero of the original moment map, hence the fibre over it has a cscK metric. Since $(\cX_{b'}, [\omega_{\uep,b'}]) \cong (\cX_b, [\omega_{\uep,b}])$, we therefore obtain the required cscK metric. We refer to \cite{ortu24} for the complete details on this argument.
\end{remark}

\section{An example}
\label{sec:examples}

Consider the toric surface $\P^1\times \P^1$ with torus action 
$$
\begin{array}{ccc}
(\C^*)^2 \times \P^1\times \P^1 & \to & \P^1\times \P^1 \\
(\alpha, \beta), ([x_1, y_1] , [x_2, y_2]) & \mapsto & ([\alpha x_1, y_1] , [\beta x_2, y_2]).
\end{array}
$$
Denote by $\check \om$ the product of the Fubini--Study metrics on each factor $\P^1$. Let $\pi : X_0\to \P^1\times \P^1$ be the blow-up at the four torus fixed points. By the work of Arezzo--Pacard (\cite{arezzo-pacard1,arezzo-pacard2}), we know that for small $\delta>0$ , $X_0$ carries a cscK metric in the class
$$
[\om_0]:=[\pi^*\check \om]-\delta(E_{0,0} +  E_{\infty,0}+ E_{0,\infty}+ E_{\infty,\infty}),
$$
wher $E_{i,j}$ is (the Poincar\'e dual of) the exceptional divisor associated to the blow-up of the point $(i,j)\in\P^1\times\P^1$. 

From \cite{rt}, following works of Ilten and Vollmert \cite{iv}, we know that the reduced automorphism group of $X_0$ is isomorphic to the $2$-dimensional torus, acting on $H^1(X_0,T^{1,0})\simeq H^{0,1}_{\om_0}(X_0,T^{1,0})\simeq \C^4$ through the representation
$$
\left[ 
\begin{array}{cccc}
                                          \alpha & 0 & 0 & 0 \\
                                            0  & \alpha^{-1} & 0 & 0 \\
                                             0 & 0 & \beta & 0 \\
                                               0 & 0 & 0 & \beta^{-1}
                                         \end{array} 
\right],
$$
where the choice of coordinates $b=(b_1,\ldots, b_4)\in \C^4$ is such that the complex deformation associated to $b\in B$, for $B$ a small neighbourhood of the origin, corresponds to the blow-up of $\P^1\times\P^1$ along the points $([b_1:1],[b_3:1])$, $([1:b_2],[1:b_4])$,  $(0,\infty)$ and $(\infty,0)$. As $X_0$ is toric, $h^{0,1}(X_0)=h^{0,2}(X_0)=0$ (\cite[Section 3.3]{Oda}), and as it is a toric surface, $H^2(X_0,TX_0)$ vanishes (see \cite[Corollary 1.5]{ilten}). We can consider the polarised complex deformation $(X, [\om_0])$ corresponding to the point $b=(0, 0, b_3, 0)$, for small $b_3$. Clearly, $(X, [\om_0])$ is analytically $K$-semistable, and degenerates to $(X_0, [\om_0])$ via $\beta \mapsto (0, 0 , \beta\, b_3 ,0)$. The stabiliser of $b$ in the torus is the first factor $\C^*\times \lbrace 1 \rbrace$. 

Consider the class 
$$
[\om_\ep]:=[\om_0]-\ep ( E_{0,\infty}+ E_{\infty,\infty}).
$$
From our construction, it follows that $(X, [\om_\ep])$ will admit a cscK metric for all $\vert \ep\vert \ll 1$ if and only if the following two conditions are satisfied:
$$
\left\{ 
\begin{array}{ccc}
\mathrm{Fut}_{[\om_\ep]}(X_0, v_\alpha) & = & 0, \\
\mathrm{Fut}_{[\om_\ep]}(X_0, v_\beta) & > & 0,
\end{array}
\right.
$$
where $\mathrm{Fut}_{[\om_\ep]}$ stands for the Futaki invariant, $ v_\alpha$ is a generator for the $\C^*\times \lbrace 1 \rbrace$-action and $ v_\beta$ is a generator for the $\lbrace 1 \rbrace\times \C^*$-action. We can compute those invariants by using the polytope description of $(X_0,[\om_\ep])$, pictured in \cite[Figure 1 page 468]{ClTip}, for $\ep >0$. 

The vanishing of $\mathrm{Fut}_{[\om_\ep]}(X_0, v_\alpha)$ follows from the symmetry along the $x$-axis, as the affine function corresponding to the hamiltonian of $ v_\alpha$ is, up to scale, given by $(x,y)\mapsto x.$ On the other hand, a direct computation, using the toric description of the Futaki invariant (\cite[Lemma 3.2.9]{donaldson}), shows that $\mathrm{Fut}_{[\om_\ep]}(X_0, v_\beta)<0$ for $\ep >0$ and $\mathrm{Fut}_{[\om_\ep]}(X_0, v_\beta)>0$ for $\ep <0$. This provides an example of wall crossing for $(X, [\om_\ep])$ that is cscK for $\ep <0$, analytically $K$-semistable for $\ep=0$ and $K$-unstable for $\ep >0$ small enough. On the other hand, $(X_0, [\om_\ep])$ is cscK precisely at $\ep = 0$.

\appendix
\section{Proof of Proposition \ref{prop:sigma-eta-image-orbit}}
\label{sec:appendix}

\begin{proof}[Proof of Proposition \ref{prop:sigma-eta-image-orbit}]
First, recall from \cite[Theorem 3.2 and Section 7.1]{Dwivedi} that $B$ is chosen so that $\mu_{\uep}$ is given by
\begin{equation}
 \label{eq:expressionmuep}
\mu_{\uep}(b')=\mu_{\uep}(0)+\sum_i \vert\vert b_i'\vert\vert^2_{\uep} \cdot m_i,
\end{equation}
for some norm $\vert\vert \cdot \vert\vert_{\uep}$ that might depend on ${\uep}$. From Equation \eqref{eq:action}, this implies in particular that for all $b'\in \cZ$, $\mu_{\uep}(b')\in \Int(\sigma)$, the interior of $\sigma$. Similarly, if $b'$ is in the boundary of $\ocZ$, then its image $\mu_{\uep}(b')$ is in the boundary of $\sigma$. To see this, let $b'\in\ocZ$ be a boundary point, so that there is $v\in \t$ such that 
$$
b'=\lim_{t\mapsto 0} \sum_i t^{\langle m_i, v \rangle}b_{m_i}.
$$
As the limit exists, $v\in\sigma^\vee$. Then
$$
\mu_{\uep}(b')=\mu_{\uep}(0)+\sum_{\langle v , m_i\rangle=0} \vert\vert b_{m_i} \vert\vert_{\uep}^2 \cdot m_i.
$$
This is a boundary point of $\sigma$, unless $\lbrace m_i,\: \langle m_i, v \rangle =0\rbrace$ spans $\mathrm{Vect}(\sigma)=\t^*$. This is impossible, as this would imply $v\in (\t^*)^\perp=\lbrace 0 \rbrace$.

The proof is now done in $3$ steps.

{\bf Step 1 :  The restriction of $\mu_{\uep}$ to $\cZ$ is a submersion.} Let $b'\in \cZ$. From the fact that $b'$ has trivial stabiliser we deduce that the map 
$$
\begin{array}{ccc}
\t\oplus i\t &  \to  & T_{b'}\cZ\\
v + i w & \mapsto & \check v(b') + i \check w (b')
\end{array}
$$
is an isomorphism, where for $v\in\t$, $\check v$ denotes the associated vector field on $B$.  We then identify 
$$
T_{b'}\cZ \simeq \t \oplus i \t
$$
and 
$$
T_{\mu_{\uep}(b')}\t^*\simeq \t^*,
$$
so that the differential of $\mu_{\uep}$ restricted to the complexified orbit at $b'$ is a map 
$$
d\mu_{\uep}(b') : \t \oplus i\t \to \t^*
$$
given by
$$
d\mu_{\uep}(b')(v+iw)=\Om_{\uep}(b')((\check v+i\check w) (b'), \cdot ).
$$
From ($R_4$), the form $\Om_0(b')$ is non-degenerate on $\lbrace \check v+i\check w (b'), \;v+iw\in\t\oplus i\t \rbrace$. Up to shrinking $U$, we can assume that $\Om_{\uep}(b')( \cdot , \cdot)$ is non-degenerate on this space as well, so that $d\mu_{\uep}$ has maximal rank on $\cZ$. Then, $\mu_{\uep}$ restricted to $\cZ$ is a submersion, and thus in particular is an open map.

{\bf Step 2 : Each ray $\rho\in\sigma(1)$ intersects  $\mu_{\uep}(\ocZ)-\mu_{\uep}(0)$.} Indeed, such a ray $\rho$ can be written $\rho = \R_+\cdot m_i$ with $H_{m_i}\cap \sigma^\vee$ a facet of $\sigma^\vee$. Pick now $v\in H_{m_i}\cap \sigma^\vee$ such that $v$ doesn't belong to any of the proper faces of $H_{m_i}\cap \sigma^\vee$. Then, by the choice of $v$,
$$
\left\{
\begin{array}{cccc}
 \langle m_i, v \rangle & = & 0,& \\
 \langle m_j, v \rangle & \geq &  0, & j\neq i.
\end{array}
\right.
$$

Assume there is a $j\neq i$ such that $\langle m_j, v \rangle =0$. Then $v\in \sigma^\vee\cap H_{m_i}\cap H_{m_j}$. As we assumed that $v$ was not in a proper face of $\sigma^\vee\cap H_{m_i}$, necessarily $\sigma^\vee\cap H_{m_i}\cap H_{m_j}=\sigma^\vee\cap H_{m_i}$. As $\sigma^\vee$ is maximal dimensional, and as $\sigma^\vee\cap H_{m_i}$ is a facet of $\sigma^\vee$ (because $\R_+\cdot m_i$ is a ray), $m_j$ must be a positive multiple of $m_i$. Hence,
$$
\left\{
\begin{array}{ccccc}
 \langle m_j, v \rangle & = & 0 & \textrm{ if }& m_j\in \R_+\cdot m_i \\
 \langle m_j, v \rangle & > & 0 & \textrm{ if } & m_j\notin \R_+\cdot m_i.
\end{array}
\right.
$$
From the action (\ref{eq:action}) of $T^\C$ on $B$ at $b$, we have that the limit point under the action generated by $v$ starting from $b$ is
$$
\lim_{t\mapsto 0} \sum_i t^{\langle m_i, v \rangle}b_{m_i}=\sum_{m_j\in\R_+\cdot m_i} b_{m_j} \in \ocZ.
$$
But from (\ref{eq:expressionmuep}) the image of that point by $\mu_\ep$ is of the form 
$$
\mu_\ep(0)+\sum_{m_j\in\R_+\cdot m_i} \lambda_j m_j=\mu_\ep(0)+\lambda m_i
$$
for positive $\lambda_j$'s and $\lambda >0$. This proves the claim of Step $2$.

{\bf Step 3 :  Conclusion of the proof.} From Step $2$, $\mu_{\uep}(\ocZ)$ contains an element of the form $\mu_\ep(0)+\lambda_i m_i$, $\lambda_i >0$, for each ray $\R_+\cdot m_i$ of $\sigma$. It also contains $\mu_{\uep}(0)$ as $0\in \overline{T^\C\cdot b}$. Denote by $\Delta$ the convex hull of $\mu_\ep(0)$ and the points $(\mu_\ep(0)+\lambda_i m_i)_{\R_+\cdot m_i\in \sigma(1)}$. It is enough to show that $\Delta\subset \mu_{\uep}(\ocZ)$ to conclude, as for $\eta$ small enough, 
$\mu_{\uep}(0)+\sigma_{\eta}\subset\Delta$.

 To prove that $\Delta\subset \mu_{\uep}(\ocZ)$, it is enough, by convexity, to prove that for any $(x_1, x_2)\in \Delta\cap \mu_{\uep}(\ocZ)$, the segment $[x_1, x_2]$ lies in $\Delta\cap \mu_{\uep}(\ocZ)$. This can be done by showing that
$$
\lbrace t\in [0, 1],\: t x_1 + (1-t) x_2 \in \mu_{\uep}(\cZ) \rbrace
$$
is open and closed. Step $1$ implies that this set open. As $\Delta$ is closed, there is a closed ball $D$ such that $(\mu_{\uep})^{-1}(\Delta)\subset D \subset B$.  Note that for any sequence $(b_i)\in T^\C\cdot b\cap D$, we can extract a converging subsequence with limit in $\overline{T^\C\cdot b}\cap D \subset \ocZ$. A classical argument using extractions of convergent subsequences and continuity of $\mu_{\uep}$ then implies the desired closedness.
\end{proof}

\end{document}